\documentclass[11pt,amsfonts]{article}
\usepackage[utf8]{inputenc}
\usepackage[T1]{fontenc} 
\usepackage{tikz}
\usepackage[english]{babel} 
\usepackage{mathrsfs}
\usepackage{amssymb}
\usepackage[cmex10]{amsmath}
\usepackage{amsthm}
\usepackage{wrapfig}
\usepackage{latexsym}
\usepackage{amsfonts}
\usepackage{color}
\usepackage[font=small, labelfont=bf]{caption}
\usepackage{ctable}
\usepackage{microtype}
\usepackage{floatrow}
\usepackage[colorlinks=true, citecolor=red, linkcolor=blue]{hyperref}
\usepackage{graphicx}
\usepackage{subfig}
\usepackage{verbatim}
\usepackage{epstopdf}
\usepackage{float}
\usepackage[symbol]{footmisc}
\usepackage{lipsum}

\usepackage[top=2cm , bottom=2cm , left=2.5cm ,right=2.5cm ]{geometry}

\usepackage{dsfont}

\usepackage{cases}

\theoremstyle{plain}
\newtheorem{definition}{Definition}
\newtheorem{theorem}{Theorem}
\newtheorem{corollary}{Corollary}

\newtheorem{example}{Example}
\newtheorem{lemma}{Lemma}

\newtheorem{remark}{Remark}

\numberwithin{equation}{section}
\numberwithin{theorem}{section}
\numberwithin{proposition}{section}
\numberwithin{definition}{section}
\numberwithin{remark}{section}
\numberwithin{corollary}{section}
\numberwithin{lemma}{section}
\numberwithin{example}{section}

\usepackage{graphicx}
\usepackage{footmisc}

\makeatletter
\newcommand*\bigcdot{\mathpalette\bigcdot@{.5}}
\newcommand*\bigcdot@[2]{\mathbin{\vcenter{\hbox{\scalebox{#2}{$\m@th#1\bullet$}}}}}
\makeatother

\begin{document}

	\sloppy
	
	\begin{center}
		{\Large \textbf{On the Existence of Optimal Controls for Reflected McKean–Vlasov Stochastic Differential Equations}} \\[0pt]
		~\\[0pt]  Ayoub Laayoun, Badr Missaoui 
        \renewcommand{\thefootnote}{}

        \footnotetext{Moroccan Center for Game Theory, Mohammed VI Polytechnic University, Rabat, Morocco.}
        \footnotetext{E-mails: \text{ayoub.laayoun@um6p.ma, badr.missaoui@um6p.ma}}
        	\renewcommand{\thefootnote}{\arabic{footnote}}

	\end{center}
	
	\renewcommand{\thefootnote}{\arabic{footnote}}

\begin{abstract}
The existence of optimal controls is a central issue in control theory and underpins many of its further developments. This paper studies optimal control problems for systems governed by  Reflected McKean-Vlasov Stochastic Differential Equations. Using the compactification method, we establish the existence of an optimal relaxed control. Under suitable convexity assumptions, we further demonstrate that this relaxed optimal control can be attained by a strict control. In addition, we show that, in general, any optimal relaxed control can be approximated by a sequence of strict controls.
\end{abstract}

\textbf{AMS Subject Classiﬁcation:}  93E20, 60G07, 60H20.

\section{Introduction }
\label{sec1}
In this paper, we investigate the existence of optimal controls for systems whose dynamics are governed by a Reflected McKean-Vlasov SDEs (RMVSDE). Specifically, we consider a state process $ X_t $ that evolves according to  
\begin{equation} \label{eq1.1}
\begin{cases}
X_t = x + \displaystyle\int_0^t b(s, X_s, \mathcal{L}_{X_s}, u_s) \, \mathrm{d}s + \displaystyle\int_0^t \sigma(s, X_s, \mathcal{L}_{X_s}) \, \mathrm{d}B_s + K_t, & t \in [0, T], \\[8pt]
X_t \ge 0 \quad \text{a.s.}, \\[4pt]
\displaystyle\int_0^T X_t \, \mathrm{d}K_t = 0 \quad \text{a.s.},
\end{cases}
\end{equation}
on a filtered probability space $ (\Omega, \mathcal{F}, \{\mathcal{F}_t\}, \mathbb{P}) $. Here, $\mathcal{L}_{X_t}$ denotes the distribution of $X_t$, $ B_t $ is a standard Brownian motion, $ x \ge 0 $ denotes the initial condition, and the functions $ b $ and $ \sigma $ specify the drift and diffusion coefficients, respectively. The process $ K_t $ acts as the reflecting term ensuring that $ X_t$  remains nonnegative. The control variable $ u = (u_t)_{t \in [0,T]} $, referred to as a \emph{strict control}, is an $ \mathcal{F}_t $-adapted measurable process taking values in a compact metric space $ A$.\\
The solution of \eqref{eq1.1} consists of a pair $ (X_t, K_t) $, where $ X_t $ is a nonnegative process and $ K_t $ is a continuous process that enforces the reflection of $ X_t $ within the domain $ [0, \infty) $. The process $ K_t $ increases only when $ X_t $ reaches the boundary at zero, thereby satisfying the Skorokhod condition $ \displaystyle\int_0^T X_t \, \mathrm{d}K_t = 0 $.

The cost function, which we aim to minimize, is defined by
\begin{equation}
    J(U) := \mathbb{E} \left[ \int_0^T \int_A f(s, X_s, \mathcal{L}_{X_s}, u_s) \, \mathrm{d}s + \int_0^T c(s, X_s, \mathcal{L}_{X_s}) \, \mathrm{d}K_s + g(X_T, \mathcal{L}_{X_T}) \right].
\end{equation}
An $ \mathcal{F}_t $-adapted control $ \hat{u} $ is called optimal if it minimizes $ J $, that is,
\begin{equation}
J(\hat{u}) = \inf_{u \in \mathcal{U}} J(u),
\end{equation}
where $ \mathcal{U} $ denotes the class of admissible controls, consisting of adapted processes with values in the action space $ A $.

In the classical framework, where the state process $ X $ evolves according to a stochastic differential equation (SDE) without reflection, stochastic control problems have been widely investigated over the past two decades using both the dynamic programming approach and the Pontryagin maximum principle (see, e.g., \cite{PengSMP, becker, Kushner-Nec, bism2, Elliot1, BMK, el karoui-Dyn, bism1, haussman1}). A well-known difficulty in this setting is that an optimal control may fail to exist within the standard class of admissible controls when the Filippov convexity condition is not satisfied. This issue stems from the fact that the admissible control set $ \mathcal{U} $ is often too restrictive, typically noncompact, so it may not contain a minimizer.
To overcome this limitation, an extended class of \emph{relaxed controls} $ \mathcal{R} $ is introduced. In this formulation, the controller is allowed, at each time $ t $, to select a probability measure $ q_t(\mathrm{d}a) $ over the control set $A $, rather than a single deterministic control $ u_t \in A $. The existence of optimal relaxed controls for SDEs with an uncontrolled diffusion coefficient was first established through compactification techniques (see \cite{kushner1, haussman-lep}). Later, El Karoui et al.~\cite{elkaroui1} extended these results to SDEs with a controlled diffusion coefficient. By reformulating the problem as a martingale problem on an appropriate canonical space, they introduced the relaxed control framework and proved a strong approximation theorem, showing that the relaxed control problem can be approximated by a sequence of strict control problems. Following this line of work, Mezerdi~\cite{Amin} established the existence of optimal controls for controlled McKean--Vlasov SDEs, using similar compactification methods as in the aforementioned studies.

The problem of controlling reflected stochastic differential equations (SDEs) arises naturally in various applications, such as queueing systems, where reflection represents physical or operational constraints. In the Markovian setting, this problem has been studied extensively, as in \cite{ContRef2}, where the control acts on the drift of the process and the objective is to minimize the long-run average cost---an ergodic or stationary control problem. In contrast, the non-Markovian framework considered in \cite{ContRefl1} addresses the stochastic optimal control of reflected path-dependent SDEs using the dynamic programming principle. There, the associated value function is characterized by a backward stochastic partial differential equation (BSPDE) with Neumann boundary conditions, which in turn provides the foundation for constructing the optimal feedback control. Related developments also appear in \cite{Ferrari}, which studies singular stochastic control problems for reflected diffusions.

The \emph{reflected McKean--Vlasov optimal control problem} plays an important role in many practical contexts. Several key models in finance, biology, cybernetics, and other applied fields naturally lead to such problems. The question of the \emph{existence of an optimal control} for reflected McKean--Vlasov systems has attracted growing attention. In \cite{Fangfang}, the authors studied the stochastic control problem for a one-dimensional reflected McKean--Vlasov SDE, characterized the associated value function through a backward stochastic partial differential equation (BSPDE), and explicitly constructed the optimal control. In \cite{Shao}, the stochastic control problem for a multi-dimensional reflected McKean--Vlasov SDE was analyzed using the theory of viscosity solutions to Hamilton--Jacobi--Bellman (HJB) equations on the Wasserstein space, formulated in terms of the intrinsic derivative. The author proved that the value function is a viscosity solution to the corresponding HJB equation on the Wasserstein space.\\
The two approaches developed in \cite{Fangfang} and \cite{Shao} rely on several structural assumptions. In both works, the diffusion coefficient $ \sigma $ satisfies a super-parabolicity condition. In \cite{Fangfang}, the functions $ f$, $c $, and $ g $ are assumed to be locally Lipschitz continuous in $ (x, \mu) $, while in \cite{Shao}, the drift coefficient $ b $ is Lipschitz continuous in $ (x, u, \mu) $, and the functions $ f $ and $ g$ are Lipschitz continuous in $ (x, u) $.

In the present paper, we adopt the \emph{compactification method} introduced in \cite{kushner1, haussman-lep, Amin} to establish the \emph{existence of an optimal relaxed control} for systems governed by one-dimensional reflected McKean--Vlasov SDEs. Our approach does not require the restrictive assumptions imposed in the aforementioned works; it only relies on the \emph{Lipschitz continuity} of $ b $ , $ \sigma $ and $f$ in $(x, \mu)$. Moreover, when the \emph{Roxin convexity condition} is satisfied, we show that the optimal relaxed control coincides with a strict control. Finally, by applying the \emph{chattering lemma}, we demonstrate that the optimal relaxed control can be approximated by a sequence of strict controls.

The paper is organized as follows. Section~\ref{sec1} reviews the concepts of strict and relaxed controls for the RMVSDE system, along with the preliminaries and assumptions underlying the model. Section~\ref{sec2} presents the main results, while Section~\ref{sec3} is devoted to the proofs.

\section{Formulation of the problem, notations and assumptions}\label{sec1}

Throughout this paper, we denote by $\mathcal{C}([0,T], \mathbb{R}^+)$ the space of continuous functions from $[0,T]$ into $\mathbb{R}^+$, equipped with the topology of uniform convergence. 

We also denote by $\mathcal{P}_2(\mathbb{R}^+)$ the space of probability measures on $\mathbb{R}^+$ with finite second-order moment, endowed with the topology induced by the 2-Wasserstein distance. For any $\mu, \nu \in \mathcal{P}_2(\mathbb{R}^+)$, the 2-Wasserstein distance $W_2(\mu, \nu)$ is defined by  
\begin{equation*}
W_2(\mu, \nu) = \inf_{\pi \in \Pi(\mu, \nu)} \left( \int_{\mathbb{R}^+ \times \mathbb{R}^+} |x - y|^2 \, d\pi(x, y) \right)^{1/2},
\end{equation*}
where $\Pi(\mu, \nu)$ denotes the set of probability measures on $\mathbb{R}^+ \times \mathbb{R}^+$ whose first and second marginals are $\mu$ and $\nu$, respectively.  

Moreover, if $\mu = P_X$ and $\nu = P_Y$ are the distributions of $\mathbb{R}^+$-valued random variables $X$ and $Y$ with finite second moments, then  
\begin{equation*}
W_2(\mu, \nu)^2 \leq \mathbb{E}\,|X - Y|^2.
\end{equation*} 
\subsection{Strict control}
We introduce the concept of strict control corresponding to system \eqref{eq1.1}.

\begin{definition}\label{def of Strict Control}
We call an admissible strict control a set $\alpha = (\Omega,\mathcal{F},\mathbb{F},\mathbb{P}, B, u, X, K)$ such that
\begin{itemize}
\item[i)]  $(\Omega,\mathcal{F},\mathbb{P})$ is a probability space equipped with a filtration $\mathbb{F}=(\mathcal{F}_t)_{t\geq 0}$;
\item[ii)] $u$ is a $A$-valued process, $\mathbb{F}$-progressively measurable;
\item[iii)]  $ B$  is a standard Brownian motion on the filtered probability space $(\Omega, \mathcal{F}, \mathbb{F}, \mathbb{P})$;
\item[iv)] $X$ and $K$ are an $\mathbb{F}$-adapted processes in $\mathbb{C}([0,T], \mathbb{R}^{+})$ satisfies the following;
\begin{equation} \label{eq1}
\begin{cases}
X_t = x + \displaystyle\int_0^t b(s, X_s,\mathcal{L}_{X_s}, u_s) \, \mathrm{d}s + \displaystyle\int_0^t \sigma(s, X_s, \mathcal{L}_{X_s}) \, \mathrm{d}B_s + K_t, & t \in [0, T], \\[8pt]
X_t \ge 0 \quad \text{a.s.}, \\[4pt]
\displaystyle\int_0^T X_t \, \mathrm{d}K_t = 0 \quad \text{a.s.}
\end{cases}
\end{equation}
\end{itemize}
We denote by $\mathcal{U}$ the set of all admissible controls.
\end{definition}
The cost functional corresponding to a control $\alpha \in \mathcal{U}$ is defined as the following
\begin{equation*}
J(\alpha) := \mathbb{E}^{\mathbb{P}}\left(\int_0^T  f(s, X_s, \mathcal{L}_{X_s}, u_s)\mathrm{d}s + \int_0^T c(s, X_s, \mathcal{L}_{X_s})\mathrm{d}K_s + g(X_T, \mathcal{L}_{X_T})\right)
\end{equation*} 
The goal is to minimize the cost functional $J$ over the set of admissible controls \( \mathcal{U} \); that is, to find an optimal control $ \alpha^* \in \mathcal{U} $ such that
$$
J(\alpha^*) = \min_{u \in \mathcal{U}} J(u).
$$

Let us assume the following conditions:
\begin{description}
\item[(A1)] Assume that the functions
\begin{equation*}
b : [0, T] \times \mathbb{R} \times \mathcal{P}_2(\mathbb{R}^+) \times A  \to \mathbb{R},\quad
\sigma : [0, T] \times \mathbb{R} \times \mathcal{P}_2(\mathbb{R}^+)\to \mathbb{R}^m
\end{equation*}
are continuous. Moreover, assume that there exists a constant $ C_1 > 0 $ such that for every  $(t, x,\mu, a) \in [0, T] \times \mathbb{R} \times \mathcal{P}_2(\mathbb{R}^+)\times A $:
\begin{equation}
|b(t, x,\mu, a)|^2 + |\sigma\sigma^*(t, x, \mu)|   \leq C_1 \left (1 + |x|^2+ \int_{\mathbb{R}^+}^{} |y|^2 \mu(\mathrm{d}y) \right) . \\
\end{equation}
\item[(A2) ] There exists a constant $ C_2 > 0 $ such that for every $ (t,a) \in [0, T] \times A$, every $ x, x' \in \mathbb{R}$ and every $ \mu, \mu' \in \mathcal{P}_2(\mathbb{R}^+)$:
\begin{equation*}
|b(t, x,\mu, a) - b(t, x',\mu' a)|+|\sigma(t, x) - \sigma(t, x') |  \leq C_2\left(|x - x'|+W_2(\mu, \mu') \right).
\end{equation*}
\item[(A.3)] 
The functions
\begin{equation*}
f : [0, T] \times \mathbb{R}^+ \times \mathcal{P}_2(\mathbb{R}^+) \times A \to \mathbb{R}, 
\quad 
h : [0, T] \times \mathbb{R}^+ \times \mathcal{P}_2(\mathbb{R}^+) \to \mathbb{R}, 
\quad 
g : \mathbb{R}^+ \times \mathcal{P}_2(\mathbb{R}^+) \to \mathbb{R}
\end{equation*}
are continuous. Moreover, there exists a constant $ C_3 > 0$  such that, for all 
$ (t, x,x' ,\mu,\mu', a) \in [0, T] \times \mathbb{R}^+ \times \mathbb{R}^+\times \mathcal{P}_2(\mathbb{R}^+) \times \mathcal{P}_2(\mathbb{R}^+) \times A $ ,
the following  conditions holds:
\begin{equation*}
|f(t, x, \mu, a)| + |h(t, x, \mu)| + |g(x, \mu)| 
\le C_3 \left( 1 + |x|^2 + \int_{\mathbb{R}^+} |y|^2 \, \mu(\mathrm{d}y) \right).
\end{equation*} 
\begin{equation*}
|f(t, x, \mu, a)-f(t, x', \mu', a)| \leq C_3\left(|x-x'|+W_{2}(\mu, \mu')\right)
\end{equation*}
\end{description}

\begin{remark}
\begin{itemize}
\item
The set $\mathcal{U}$ is nonempty. Indeed, let $(\Omega, \mathcal{F}, \mathbb{F}, \mathbb{P})$ be a filtered probability space, and let $B$ be a standard Brownian motion defined on this space.  
Under Assumptions~$(\mathbf{A_1})$ and~$(\mathbf{A_2})$, Theorem~3.5 in~\cite{Adams} (for the domain $D = [0, \infty)$) guarantees that, for any fixed constant control $u_0 \in A$, the following system:
\begin{equation*}
\begin{cases}
X_t = x + \displaystyle\int_0^t b(s, X_s, \mathcal{L}_{X_s}, u_0) \, \mathrm{d}s 
      + \displaystyle\int_0^t \sigma(s, X_s, \mathcal{L}_{X_s}) \, \mathrm{d}B_s 
      + K_t, & t \in [0, T], \\[8pt]
X_t \ge 0, \quad \text{a.s.}, \\[4pt]
\displaystyle\int_0^T X_t \, \mathrm{d}K_t = 0, \quad \text{a.s.}
\end{cases}
\end{equation*}
admits a unique strong solution $(X^{u_0}, K^{u_0})$.\\  
Define the constant control process $u_t := u_0$ for all $t \in [0, T]$. Then the tuple $(\Omega, \mathcal{F}, \mathbb{F}, \mathbb{P}, B, u, X^{u_0}, K^{u_0})$ belongs to the set of admissible controls $\mathcal{U}$. Hence, $\mathcal{U}$ is nonempty.
\item
Moreover, by Theorem~3.5 in~\cite{Adams} and under Assumptions~$(\mathbf{A_1})$–$(\mathbf{A_3})$, the function $J$ is well-defined.
\item In general, the existence of an optimal strict control may fail even under assumptions \textbf{(A\(_1\)–A\(_3\))} (see Example~\ref{Example3}). In such cases, the set of strict controls is extended to a broader class of measure-valued controls, known as \emph{relaxed controls}, whose definition will be given in the next subsection.
\end{itemize}
\end{remark}
\subsection{Relaxed control}
The concept of \emph{relaxed control} consists in replacing the $A$-valued process $(u_t)$ by a $P(A)$-valued process $(q_t)$, where $P(A)$ denotes the space of probability measures on $A$, endowed with the topology of weak convergence.

\begin{itemize}
\item We denote by $\mathcal{V}$ the set of probability measures $q$ on $[0, T] \times A$ satisfying
\begin{equation*}
q([0, t] \times A) = t, \quad \forall\, t \in [0, T].
\end{equation*}
\item A \emph{relaxed control} can therefore be viewed as a random variable taking values in $\mathcal{V}$ (\cite{elkaroui2}).
\item Endowed with the topology of stable convergence of measures, the space $\mathcal{V}$ is compact and metrizable, as established by Jacod and Mémin \cite{jacod1}.
\end{itemize}

We now extend Definition~\ref{def of Strict Control} to the framework of relaxed controls.
\begin{definition}\label{def of relaxed control}
A tuple $r = (\Omega, \mathcal{F}, \mathbb{F}, \mathbb{P}, B, q, X, K)$ is called \emph{an admissible relaxed control} if it satisfies Conditions $i)$, and $iii)$ of Definition~\ref{def of Strict Control}, and the following:

\begin{itemize}
    \item[ii')] $q$ is $\mathcal{V}$-valued random variable such that for each $t\in[0,T]$, $\mathds{1}_{]0,t]}q$ is $\mathcal{F}_t$- measurable;
    \item[iv')] $X$ and $K$ are an $\mathbb{F}$-adapted processes in $\mathbb{C}([0,T], \mathbb{R}^+)$ satisfy the following
\begin{equation} \label{def of Relaxed}
\begin{cases}
\mathrm{d}X_t =\displaystyle\int_A b(t, X_t,\mathcal{L}_{X_t}, a) q_t(\mathrm{d}a)\mathrm{d}t +  \sigma(t, X_t, \mathcal{L}_{X_t}) \mathrm{d}B_t + \mathrm{d} K_t, & t \in [0, T], \\
X_t \geq 0 \quad \text{a.s}, \\
\displaystyle\int_0^T X_t \mathrm{d}K_t = 0 \quad \text{a.s.}
\end{cases}
\end{equation}
\end{itemize}
We denote by $\mathcal{R}$ the set of all relaxed controls.
\end{definition}
\begin{remark}
Under Assumptions $(\mathbf{A_1} - \mathbf{A_2})$, the set $\mathcal{R}$ of relaxed controls is nonempty. In fact, by choosing $q_t = \delta_{a_0}$ for some fixed $a_0 \in A$ and applying Theorem 3.5 from~\cite{Adams}, the existence of such a relaxed control follows immediately.
\end{remark}

The cost functional corresponding to a relaxed control $r \in \mathcal{R}$ is defined as follows:
\begin{equation*}
J(r) := \mathbb{E}^{\mathbb{P}}\left(\int_0^T \int_{A} f(s, X_s, \mathcal{L}_{X_s}, a) q_s(\mathrm{d}a)\mathrm{d}s + \int_0^T c(s, X_s, \mathcal{L}_{X_s})\mathrm{d}K_s + g(X_T, \mathcal{L}_{X_T})\right).
\end{equation*} 
\section{The main results}\label{sec2}
In this section, we present the main results of the paper: the existence of an optimal relaxed control (Theorem~\ref{theo1}), the existence of a strict optimal control under a convexity condition (Theorem~\ref{theo2}), and approximation results in the general case (Theorem~\ref{approximatio1}).
  
\begin{theorem} \label{theo1}
    Under assumptions $\textbf{(A.1)-(A.3)}$, the relaxed control problem has an optimal solution.
\end{theorem}
To address the existence of a strict optimal control, we need the Roxin condition, given by \\
\textbf{(A.4)} For every $(t,x,\mu) \in [0,T] \times \mathbb{R} \times \mathcal{P}_2(\mathbb{R}^+)$, the set 
\begin{equation}
    \mathcal{S}(t,x,\mu) := \{ (b(t,x,\mu,a),f(t,x,\mu,a)): \ a \in A \}
\end{equation}
is convex and closed in $\mathbb{R} \times \mathbb{R}$.\\
\begin{corollary}\label{theo2}
    If the assumptions $\textbf{(A.1)-(A.4)}$ hold, then, the relaxed control problem has a strict optimal control.
\end{corollary}

\begin{remark}
When Roxin’s condition is not satisfied, a strict optimal control may fail to exist. This is illustrated by the counterexample provided in Example~\ref{Example3}.
\end{remark}

\begin{example}[see page~86 in \cite{Kushner-exp}]\label{Example3}
Consider the following classical example.  
Let the state space be the whole real line $\mathbb{R}$, and the control set $A := [-1, 1]$.  
We work in the deterministic case with dynamics and cost functional given by
\begin{equation*}
\dot{X}(t) = u(t), \quad X(0) = 0,
\end{equation*}
and
\begin{equation*}
J(u) := \int_{0}^{T} \left[ X^2(t) + (u^2(t) - 1)^2 \right] \, \mathrm{d}t.
\end{equation*}
Note that, for each fixed $x$, the set
\begin{equation*}
\{ (a, \, x^2 + (a^2 - 1)^2) : a \in [-1, 1] \}
\end{equation*}
is not convex.  
It is easy to see that the infimum of $J(u)$ equals zero, but this infimum is not attained by any strict (ordinary) control.

Now define a sequence of relaxed controls by
\begin{equation*}
q^n(\mathrm{d}a, \mathrm{d}t) := \delta_{u^n(t)}(\mathrm{d}a) \, \mathrm{d}t,
\end{equation*}
where
\begin{equation*}
u^n(t) := (-1)^k \quad \text{if} \quad \frac{k}{n} \le t < \frac{k+1}{n}, \quad 0 \le k \le n-1.
\end{equation*}
Then the sequence $\{q^n\}$ converges weakly to
\begin{equation*}
\frac{1}{2} \left( \delta_{-1} + \delta_{1} \right)(\mathrm{d}a) \, \mathrm{d}t,
\end{equation*}
which is an optimal relaxed control.
\end{example}
As shown in the preceding example, a strict optimal control may fail to exist. However, it is possible to construct a sequence of strict (ordinary) controls that approximates the relaxed optimal control. We will then extend this result to our general control problem.

\begin{theorem}\label{approximatio1}
In general, under assumptions $\textbf{(A.1)--(A.3)}$, there exists a sequence of strict controls $(\alpha^n)_{n \ge 1}$ such that
\begin{equation*}
\inf_{r \in \mathcal{R}} J(r) = \lim_{n \to \infty} J(\alpha^n).
\end{equation*}
As a consequence,
\begin{equation*}
\inf_{r \in \mathcal{R}} J(r) = \inf_{ \alpha \in \mathcal{U} }J(\alpha).
\end{equation*}
\end{theorem}

\section{Proof of the main results}\label{sec3}
To prove the theorem \ref{theo1}, we need to some auxiliary results concerning the tightness of process.\\
Let  
\begin{equation*}
r^n = (\Omega^n, \mathcal{F}^n, \mathbb{F}^n, \mathbb{P}^n, B^n, q^n, X^n, K^n)
\end{equation*} 
be a minimizing sequence such that  
\begin{equation*}
\lim_{n \to \infty} J(r^n) = \inf_{r \in \mathcal{R}} J(r).
\end{equation*}
For each $n$, the pair $(X^n, K^n)$ denotes the solution of the reflected McKean-Vlasov SDE
\begin{equation} \label{eq2}
\begin{cases}
\displaystyle
\mathrm{d}X^n_t = \int_A b(t, X^n_t, \mathcal{L}_{X^n_t}, a)\, q^n_t(\mathrm{d}a)\, \mathrm{d}t
+ \sigma(t, X^n_t, \mathcal{L}_{X^n_t})\, \mathrm{d}B^n_t
+ \mathrm{d}K^n_t, & t \in [0, T], \\[1em]
X^n_t \ge 0, \quad \text{a.s.}, \\[0.5em]
\displaystyle \int_0^T X^n_t\, \mathrm{d}K^n_t = 0, \quad \text{a.s.}
\end{cases}
\end{equation}

In what follows, the constant $ C $ may vary from line to line. We denote by $\sigma^* $ the transpose of $ \sigma $.
\begin{lemma}\label{lemma1}
There exists a positive constant C such that
\begin{equation} \label{eq4.2}
\sup_{n} \, \mathbb{E}^{\mathbb{P}^n}\!\left[ 
   \sup_{0 \leq t \leq T} |X^n_t|^{4} 
   + \sup_{0 \leq t \leq T} |K^n_t|^{4} 
\right] \leq C.
\end{equation}
Moreover, the processes \((X^n, K^n, B^n)\) are tight in the space 
\begin{equation*}
\mathcal{C}([0,T], \mathbb{R}^+) \times \mathcal{C}([0,T], \mathbb{R}^+) \times \mathcal{C}([0,T], \mathbb{R}^m).
\end{equation*}
\end{lemma}
\begin{proof} 
To establish the tightness of the sequence $(X^n, K^n, B^n)$, we apply Aldous’ tightness criterion stated in Lemma~\ref{Tights}. The proof of this lemma is carried out in two steps.

\textbf{Step 1:} In this step, we establish the inequality \eqref{eq4.2}.

By using Itô’s formula, we get
\begin{align*}
|X^n_t|^2 &= x^2 + 2 \int_0^t \int_A^{} b(s,X^n_s, \mathcal{L}_{X^n_s},a)X^n_s q^n_s(\mathrm{d}a)\mathrm{d}s +2 \int_0^t X^n_s \sigma(s,X^n_s,\mathcal{L}_{X^n_s})\mathrm{d}B^n_s \\
&\quad  \quad+\int_0^t \sigma \sigma^*(s,X^n_s, \mathcal{L}_{X^n_s})\mathrm{d}s
+\int_0^t X^n_s \mathrm{d}K^n_s \\
&= x^2 + 2 \int_0^t \int_A^{} b(s,X^n_s, \mathcal{L}_{X^n_s},a)X^n_s q^n_s(\mathrm{d}a)\mathrm{d}s +2 \int_0^t X^n_s \sigma(t,X^n_s, \mathcal{L}_{X^n_s})\mathrm{d}B^n_s \\
&\quad \quad +\int_0^t \sigma \sigma^*(s,X^n_s, \mathcal{L}_{X^n_s})\mathrm{d}s,
\end{align*}
where the second equality is due to the Skorokhod’s condition.\\
We take the supremum over $[0,T]$ and the expectation , we obtain
\begin{align*}
\mathbb{E}^{\mathbb{P}^n}\!\left[\sup_{0 \leq t \leq T} |X^n_t|^{4}\right] 
&\leq C\Bigg( x^{4} 
+ \mathbb{E}^{\mathbb{P}^n}\!\left[\sup_{0 \leq t \leq T}\int_0^t \int_A |b(s,X^n_s,\mathcal{L}_{X^n_s},a)X^n_s|^2\,q^n_s(\mathrm{d}a)\, \mathrm{d}s \right] \\
&\quad \quad \quad  + \mathbb{E}^{\mathbb{P}^n}\!\left[\sup_{0 \leq t \leq T}\left|\int_0^t X^n_s \sigma(s,X^n_s, \mathcal{L}_{X^n_s})\,\mathrm{d}B^n_s\right|^2\right] \\
&\quad \quad \quad + \mathbb{E}^{\mathbb{P}^n}\!\left[\sup_{0 \leq t \leq T}\int_0^t |\sigma \sigma^*(s,X^n_s, \mathcal{L}_{X^n_s})|^2\,\mathrm{d}s\right] \Bigg) \\
&\leq C\Bigg(1+x^{4} 
+ \mathbb{E}^{\mathbb{P}^n}\!\left[\int_0^T |X^n_t|^{4}\,\mathrm{d}t \right] \\
&\quad \quad \quad + \mathbb{E}^{\mathbb{P}^n}\!\left|\int_0^T |X^n_s|^2 \sigma\sigma^*(s,X^n_s, \mathcal{L}_{X^n_s})\,\mathrm{d}s\right|\Bigg) \\
&\leq C\Bigg(1+x^{4} 
+ \mathbb{E}^{\mathbb{P}^n}\!\left[\int_0^T |X^n_t|^{4}\,\mathrm{d}t \right] \Bigg).
\end{align*}
The second inequality follows from the Burkholder–Davis–Gundy inequality and \textbf{(A.2)}; the third follows from \textbf{(A.2)}.\\ 
Then by Gronwall’s lemma, there exists a constant C such that 
\begin{equation}\label{estX}
\mathbb{E}^{\mathbb{P}^n}\left[\sup_{0 \leq t \leq T} |X^n_t|^{4}\right] \leq C.
\end{equation}
On the other hand, we have 
\begin{equation*}
K^n_t = X^n_t -x-\int_0^t \int_A b(s, X^n_s,\mathcal{L}_{X^n_s}, a) q^n_s(\mathrm{d}a)\mathrm{d}s - \int_0^t \sigma(s, X^n_s, \mathcal{L}_{X^n_s}) \mathrm{d}B^n_s.
\end{equation*}
Then, by \eqref{estX}, assumption \textbf{(A.2)}, and the Burkholder--Davis--Gundy inequality, there exists a constant $C > 0$ such that 
\begin{equation*}
\mathbb{E}^{\mathbb{P}^n}\!\left[\sup_{0 \leq t \leq T} |K^n_t|^{4}\right] \leq C.
\end{equation*}

\textbf{Step 2:}
For $0 \leq s \leq u \leq T$, we consider the solution $(X^{n,s}, K^{n,s})$ of the following reflected McKean-Vlasov SDE 
\begin{equation}
\begin{cases}
X^{n,s}_t =X^n_s + \displaystyle\int_s^t \int_A b(r, X^{n,s}_r,\mathcal{L}_{X^{n,s}_r} ,a) q^n_r(\mathrm{d}a)\mathrm{d}r + \displaystyle\int_s^t \sigma(r, X^{n,s}_r, \mathcal{L}_{X^{n,s}_r}) \mathrm{d}B^n_r + K^{n,s}_t, & t \in [s, T], \\
X^{n,s}_t \geq 0 \quad \text{a.s}, \\
\displaystyle\int_{s}^{T} X^{n,s}_t \mathrm{d}K^{n,s}_t = 0 \quad \text{a.s.}
\end{cases}
\end{equation}
By proceeding as in the proof of Lemma \ref{lemma1}, we can easily get the following estimation
\begin{equation} \label{esti11}
\mathbb{E}^{\mathbb{P}^n}\left[\sup_{s \leq t \leq T} |X^{n,s}_t|^4 \right] \leq C,
\end{equation}
where the constant $C$ is independent of $n$.\\
Now, applying Itô’s formula to the continuous semimartingale $|X^n_t - X^{n,s}_t|^2$ yields to: 
\begin{align*}
\mathrm{d}|X^n_t - X^{n,s}_t|^2 &= 2 \int_A^{} (X^n_t - X^{n,s}_t)\left( b(t,X^n_t, \mathcal{L}_{X^{n}_t},a)-b(t,X^{n,s}_t,\mathcal{L}_{X^{n,s}_t},a)\right) q^n_t(\mathrm{d}a)\mathrm{d}t \\ & +2 (X^n_t - X^{n,s}_t)\left(\sigma(t,X^n_t,\mathcal{L}_{X^{n}_t})-\sigma(t,X^{n,s}_t, \mathcal{L}_{X^{n,s}_t})\right)\mathrm{d}B^n_t \\ & +\left(\sigma(t,X^n_t, \mathcal{L}_{X^{n}_t})-\sigma(t,X^{n,s}_t, \mathcal{L}_{X^{n,s}_t})\right) \left(\sigma(t,X^n_t, \mathcal{L}_{X^{n}_t})-\sigma(t,X^{n,s}_t, \mathcal{L}_{X^{n,s}_t })\right)^*\mathrm{d}t \\ 
& + 2 (X^n_t - X^{n,s}_t)\mathrm{d}K^n_t -2 (X^n_t - X^{n,s}_t)\mathrm{d}K^{n,s}_t.
\end{align*} 
The Skorokhod conditions implies that $(X^n_t - X^{n,s}_t)\mathrm{d}K^n_t \leq 0$ and\\
$(X^n_t - X^{n,s}_t)\mathrm{d}K^{n,s}_t \geq 0$, then we get,
\begin{align*}
\mathrm{d}|X^n_t - X^{n,s}_t|^2 &\leq 2 \int_A^{} (X^n_t - X^{n,s}_t)\left(b(t,X^n_t, \mathcal{L}_{X^{n}_t},a)-b(t,X^{n,s}_t,\mathcal{L}_{X^{n,s}_t},a)\right)q^n_t(\mathrm{d}a)\mathrm{d}t \\ & +2 (X^n_t - X^{n,s}_t)\left(\sigma(t,X^n_t, \mathcal{L}_{X^{n}_t})-\sigma(t,X^{n,s}_t, \mathcal{L}_{X^{n,s}_t})\right)\mathrm{d}B^n_t \\ & +\left(\sigma(t,X^n_t, \mathcal{L}_{X^{n}_t})-\sigma(t,X^{n,s}_t, \mathcal{L}_{X^{n,s}_t})\right)\left(\sigma(t,X^n_t, \mathcal{L}_{X^{n}_t})-\sigma(t,X^{n,s}_t, \mathcal{L}_{X^{n,s}_t})\right)^*\mathrm{d}t.
\end{align*}
By applying the Burkholder–Davis–Gundy inequality, the Cauchy-Schwarz inequality, and leveraging assumption \textbf{(A.2)}, along with Lemma \eqref{eq4.2} and \eqref{esti11}, we have that the local martingale
\begin{equation*}
\int_0^{.}(X^n_t - X^{n,s}_t)\left(\sigma(t,X^n_t, \mathcal{L}_{X^{n}_t})-\sigma(t,X^{n,s}_t, \mathcal{L}_{X^{n,s}_t})\right)\mathrm{d}B^n_t
\end{equation*}
is uniformly integrable martingale.\\
If we take the integral with respect to s and u, and then take the expectation, we obtain 
\begin{align*}
\mathbb{E}^{\mathbb{P}^n}\left[ |X^n_u - X^{n,s}_u|^2 \right ] & \leq 2\mathbb{E}^{\mathbb{P}^n}\left[\int_s^u \int_A^{} (X^n_t - X^{n,s}_t)\left(b(t,X^n_t, \mathcal{L}_{X^{n}_t},a)-b(t,X^{n,s}_t,\mathcal{L}_{X^{n,s}_t},a)\right)q^n_t(\mathrm{d}a)\mathrm{d}t\right ]\\
&+\mathbb{E}^{\mathbb{P}^n}\left[\int_s^u \left(\sigma(t,X^n_t, \mathcal{L}_{X^{n}_t})-\sigma(t,X^{n,s}_t,\mathcal{L}_{X^{n,s}_t})\right)\left(\sigma(t,X^n_t,\mathcal{L}_{X^{n}_t})-\sigma(t,X^{n,s}_t,\mathcal{L}_{X^{n,s}_t})\right)^*\mathrm{d}t \right].
\end{align*}
By assumption \textbf{(A.2)}, there exists a constant $C > 0$, independent of s and u,
such that 
\begin{equation}
\mathbb{E}^{\mathbb{P}^n}\left[|X^n_u - X^{n,s}_u|^2 \right] \leq C\mathbb{E}^{\mathbb{P}^n}\left[\int_s^u|X^n_t - X^{n,s}_t|^2 \mathrm{d}t\right].
\end{equation}
Then
\begin{align*}
\mathbb{E}^{\mathbb{P}^n}\left[|X^n_u - X^{n,s}_u|^2\right] &\leq C(u-s)\mathbb{E}^{\mathbb{P}^n}\left[\sup_{s \leq t \leq T}|X^n_t - X^{n,s}_t|^2\right]\\
&\leq 2C(u-s)\left(\mathbb{E}^{\mathbb{P}^n}\left[\sup_{s \leq t \leq T}|X^n_t|^2\right]+\mathbb{E}^{\mathbb{P}^n}\left[\sup_{s \leq t \leq T}|X^{n,s}_t|^2\right] \right).
\end{align*}
By the uniqueness of the solution to equation (\ref{eq2}), it follows that $X^{n,s} = X^n_{. \wedge s}$ on $[s,T]$, which implies that
\begin{equation}
\mathbb{E}^{\mathbb{P}^n}\left[|X^n_u - X^n_s|^2\right] \leq 4C(u-s)\mathbb{E}^{\mathbb{P}^n}\left[\sup_{0 \leq t \leq T}|X^n_t|^2 \right]\leq 4C(u-s)\mathbb{E}^{\mathbb{P}^n}\left[\sup_{0 \leq t \leq T}|X^n_t|^4 \right]^{1/2}.
\end{equation}
And the tightness of $(X^n)_{n \geq0}$ follows from \eqref{eq4.2}.\\
Similarly, by using the fact that 
\begin{equation}
K^n_t = X^n_t -x-\int_0^t \int_A b(r, X^n_r,\mathcal{L}_{X^{n}_r}, a) q^n_r(\mathrm{d}a)\mathrm{d}r - \int_0^t \sigma(r, X^n_r, \mathcal{L}_{X^{n}_r}) \mathrm{d}B^n_r,
\end{equation}
we can prove the tightness of $(K^n)_{n \geq 0}$. The tightness of $(B^n)_{n \geq 0}$ can be shown using standard arguments.
\end{proof}
\begin{lemma} (Lemma 3.5 \cite{bahlali1}).\label{lemma3}
The family of relaxed controls $(q^n)_{n \geq 0}$ is tight in $\mathcal{V}$.
\end{lemma}

\subsection{Proof of theorem \ref{theo1}}
\begin{proof}
Lemmas \ref{lemma1} and \ref{lemma3} implies that the sequence of processes $\gamma ^n= (X^n, B^n, K^n, q^n)$ is tight on the space\\
\begin{equation}
\Gamma := \mathcal{C}([0,T], \mathbb{R^+}) \times \mathcal{C}([0,T], \mathbb{R}^m) \times \mathcal{C}([0,T], \mathbb{R^+}) \times \mathcal{V},
\end{equation}
equipped with the product topology of the uniform convergence, and the topology of stable convergences of measures.\\
By Skorokhod representation Lemma \ref{lemma-representation}, there exists a probability space $(\hat{\Omega},\hat{\mathcal{F}}, \hat{\mathbb{P}})$, a sequence $\hat{\gamma^n} = (\hat{X^n}, \hat{B^n}, \hat{K^n}, \hat{q^n})$ and $\hat{\gamma} = (\hat{X}, \hat{B}, \hat{K}, \hat{q})$ defined on this space such that :
\begin{description}
\item[1)] for each $n \in \mathbb{N},$
\begin{equation}
\mathcal{L}(\hat{\gamma^n}) = \mathcal{L}(\gamma ^n).
\end{equation}
\item[2)] there exists a sub-sequence $(\hat{\gamma^n_k})$ of $(\hat{\gamma^n})$, denoted again by $(\hat{\gamma^n})$, it converges almost surely $\hat{\mathbb{P}}-a.s.$ to $(\hat{\gamma})$ in the space $\Gamma$.
\item[3)] $\sup_{ 0 \leq t \leq T}|\hat{X^n}_t-\hat{X}_t| \to 0$ and $\sup_{ 0 \leq t \leq T}|\hat{K^n}_t-\hat{K}_t| \to 0$ $\hat{\mathbb{P}}-a.s.$
\end{description}
By property \textbf{1)}, we drive 
\begin{equation} \label{eq5}
\begin{cases}
\mathrm{d}\hat{X^n}_t =\displaystyle\int_A b(t, \hat{X^n}_t, \mathcal{L}_{\hat{X^n}_t}, a) \hat{q^n}_t(\mathrm{d}a)\mathrm{d}t +  \sigma(t,\hat{ X^n}_t, \mathcal{L}_{\hat{X^n}_t}) \mathrm{d}\hat{B^n}_t + \mathrm{d}\hat{K^n}_t, & t \in [0, T], \\
\hat{X^n}_t \geq 0 \quad \text{a.s}, \\
\displaystyle\int_0^T \hat{X^n}_t \mathrm{d}\hat{K^n}_t = 0 \quad \hat{\mathbb{P}}-\text{a.s.} 
\end{cases}
\end{equation}

To proceed with the passage to the limit, we begin by proving the following limits

\begin{equation}
\int_{0}^{t}\int_{\mathbb{A}}b(s,\hat{X}^n_s,\mathcal{L}_{\hat{X^n}_t}, a)\hat{q}_s^n(\mathrm{d}a)\mathrm{d}s \to  \int_{0}^{t}\int_{\mathbb{A}}b(s,\hat{X}_s,\mathcal{L}_{\hat{X}_t} ,a)\hat{q}_s(\mathrm{d}a)\mathrm{d}s \label{limits_1}
\end{equation}

\begin{equation*}
\int_{0}^{T}\sigma(s,\hat{X}^n_s,\mathcal{L}_{\hat{X^n}_s} )\mathrm{d}\hat{B}_s^n \to \int_{0}^{T}\sigma(s,\hat{X}_s, \mathcal{L}_{\hat{X}_s})\mathrm{d}\hat{B}_s
\end{equation*}
We have
\begin{equation}
\begin{split}
\int_{0}^{T}\int_{\mathbb{A}}b(s,\hat{X}^n_s,\mathcal{L}_{\hat{X^n}_s}, a)\hat{q}_s^n(\mathrm{d}a)\mathrm{d}s&-  \int_{0}^{T}\int_{\mathbb{A}}b(s,\hat{X}_s,\mathcal{L}_{\hat{X}_s}, a)\hat{q}_s(\mathrm{d}a)\mathrm{d}s\\
&=\int_{0}^{T}\int_{\mathbb{A}}b(s,\hat{X}^n_s,\mathcal{L}_{\hat{X^n}_s}, a)-b(s,\hat{X}_s,\mathcal{L}_{\hat{X}_t}, a)\hat{q}_s^n(\mathrm{d}a)\mathrm{d}s \\
&+\int_{0}^{T}\int_{\mathbb{A}}b(s,\hat{X}_s,\mathcal{L}_{\hat{X}_t}, a)(\hat{q}_s^n-\hat{q}_s)(\mathrm{d}a)\mathrm{d}s\label{twotherm}
\end{split}
\end{equation}

From the Lipschitz assumptions on $b$, we get 
\begin{equation}\label{thefirst}
\hat{\mathbb{E}}\!\left[\int_{0}^{T}\!\int_{\mathbb{A}}
\big|b(s,\hat{X}^n_s,\mathcal{L}_{\hat{X^n}_t}, a)
- b(s,\hat{X}_s,\mathcal{L}_{\hat{X}_t}, a)\big|^2
\,\hat{q}_s^n(\mathrm{d}a)\,\mathrm{d}s\right]
\leq C \hat{\mathbb{E}} \left[ \int_0^T |\hat{X}_s^n-\hat{X}_s|^2\mathrm{d}s \right]
\end{equation}

Using properties \textbf{1)–3)}, together with \eqref{eq4.2} and the dominated convergence theorem, we obtain
\[
\lim_{n \to \infty} \hat{\mathbb{E}}\left[ \int_0^T \big|\hat{X}^n_t - \hat{X}_t\big|^2 \, \mathrm{d}s \right] = 0.
\]
Combining this result with \eqref{thefirst}, we conclude that the first term on the right-hand side of \eqref{twotherm} converges to zero in probability.

It remains to prove the convergence of the second term in \eqref{twotherm}. Let $M>0$, we have
\begin{equation}\label{twotherm1}
\begin{split}
\left|\int_{0}^{t}\int_{\mathbb{A}}b(s,\hat{X}_s,\mathcal{L}_{\hat{X}_s}, a)(\hat{q}_s^n-\hat{q}_s)(\mathrm{d}a)\mathrm{d}s\right|&\leq \left|\int_{0}^{t}\int_{\mathbb{A}}b(s,\hat{X}_s,\mathcal{L}_{\hat{X}_s},a)\mathds{1}_{\{|\hat{X}_{s}|< M\}}(\hat{q}_s^n-\hat{q}_s)(\mathrm{d}a)\mathrm{d}s\right|\\
&+\left|\int_{0}^{t}\int_{\mathbb{A}}b(s,\hat{X}_s,\mathcal{L}_{\hat{X}_s},a)\mathds{1}_{\{|\hat{X}_{s}| \geq M\}}(\hat{q}_s^n-\hat{q}_s)(\mathrm{d}a)\mathrm{d}s\right|.
\end{split}
\end{equation}
We have that the mapping $(s,a) \mapsto b(s,\hat{X}_s,\mathcal{L}_{\hat{X}_s},a)\mathds{1}_{\{|\hat{X}_{s}|<M\}}$ is bounded and measurable, and the function $a \mapsto b(s,\hat{X}_s,\mathcal{L}_{\hat{X}_s},a)\mathds{1}_{\{|\hat{X}_{s}|<M\}}$ is continuous. From \textbf{2)} , we conclude that the first term on the right-hand side of \eqref{twotherm1} converges to zero. Then, by the Lebesgue dominated convergence theorem, we obtain convergence in $L_1(\hat{\mathbb{P}})$. 

On the other hand, from the linear growth assumption on $b$, we obtain 
\begin{align*}
&\hat{\mathbb{E}}\left(\left|\int_{0}^{T}\int_{\mathbb{A}}b(s,\hat{X}_s,\mathcal{L}_{\hat{X}_s},a)\mathds{1}_{\{|\hat{X}_{s}|\geq M\}}(\hat{q}_s^n-\hat{q}_s)(\mathrm{d}a)\mathrm{d}s\right| \right) \\
&\leq\gamma\hat{\mathbb{E}}\left(\int_{0}^{T} \left( 1+|\hat{X}_{s}|\right) \mathds{1}_{\{|\hat{X}_{s}|\geq M\}} \mathrm{d}s \right)\\
&\leq \frac{1}{M} \int_{0}^{T} \hat{\mathbb{E}}\left(|\hat{X}_{s}|^2+ 1 \right) \mathrm{d}s 
\end{align*}
where, in the last inequality we used the formula $x\mathds{1}_{\{ x>a\}}\leq \frac{x^2}{a}$, valid for $x,a>0$. Letting $M\to\infty$, and then $n\to\infty$, we obtain that the second term in \eqref{twotherm1} converges in $L_1(\hat{\mathbb{P}})$ to zero. Thus, we conclude the first convergence in \eqref{limits_1}.\\
The second convergence follows from Lemma 3.1 in \cite{Krylov-conv}.

By applying properties \textbf{2)} and \textbf{3)} together with the previous limits, and then passing to the limit in~\eqref{eq5}, we obtain

\begin{equation} 
\begin{cases}
\mathrm{d}\hat{X}_t =\displaystyle\int_A b(t, \hat{X}_t,\mathcal{L}_{\hat{X}_t}, a) \hat{q}_t(\mathrm{d}a)\mathrm{d}t +  \sigma(t,\hat{ X}_t,\mathcal{L}_{\hat{X}_t}) \mathrm{d}\hat{B}_t + \mathrm{d}\hat{K}_t, & t \in [0, T], \\
\hat{X}_t \geq 0 \quad \text{a.s}, \\
\displaystyle\int_0^T \hat{X}_t \mathrm{d}\hat{K}_t = 0 \quad \hat{\mathbb{P}}-\text{a.s.} 
\end{cases}
\end{equation}
The convergence 
\begin{equation*}
0 = \int_0^T \hat{X}^n_t \, \mathrm{d}\hat{K}^n_t \to \int_0^T \hat{X}_t \, \mathrm{d}\hat{K}_t \quad \hat{\mathbb{P}}-\text{a.s.}
\end{equation*}
follows from property \textbf{3)} and the continuity of the mapping 
\begin{equation*}
(x,k) \mapsto \int_0^T x_t \, \mathrm{d}k_t,
\end{equation*}
under the uniform convergence topology.

To complete the proof of Theorem \ref{theo1}, it remains to verify that $\hat{q}$ is indeed an optimal control.\\
By repeating the same arguments as above and using assumption \textbf{(A.3)} together with the dominated convergence theorem, we obtain the desired convergences.
\begin{equation*}
\lim_{n \to \infty} \hat{\mathbb{E}}\left[ \int_0^T \int_A^{} f(s, \hat{X^n}_s,\mathcal{L}_{\hat{X}_s}, a)\hat{q^n}_s(\mathrm{d}a) \mathrm{d}s \right] = \hat{\mathbb{E}}\left[ \int_0^T \int_A^{} f(s, \hat{X}_s,\mathcal{L}_{\hat{X}_s}, a)\hat{q}_s(\mathrm{d}a) \mathrm{d}s \right]
\end{equation*}
\begin{equation*}
\lim_{n \to \infty} \hat{\mathbb{E}}\left[ g(\hat{X^n}_T, \mathcal{L}_{\hat{X^n}_T}) +\int_0^T c(s,\hat{X}^n_s,\mathcal{L}_{\hat{X^n}_s} )\mathrm{d}\hat{K}^n_s\right] = \hat{\mathbb{E}}\left[ g(\hat{X}_T, \mathcal{L}_{\hat{X}_T}) + \int_0^T c(s,\hat{X}_s, \mathcal{L}_{\hat{X}_s} )\mathrm{d}\hat{K}_s\right]
\end{equation*}
From properties \textbf{1)–3)} and the previous convergence results, we conclude that
\begin{align*}
\inf_{r \in \mathcal{R}} J(r) &= \lim_{n \to \infty} J(r^n) \\
&= \lim_{n \to \infty} \mathbb{E}^{\mathbb{P}^n}\left[ \int_0^T \int_A^{} f(s, X^n_s,\mathcal{L}_{X^n_s}, a)q^n_s(\mathrm{d}a) \mathrm{d}s + \int_0^T c(s, X^n_s, \mathcal{L}_{X^n_s})\mathrm{d}K^n_s + g(X^n_T, \mathcal{L}_{X^n_T}) \right]\\
&= \lim_{n \to \infty} \hat{\mathbb{E}}\left[ \int_0^T \int_A^{} f(s, \hat{X}^n_s,\mathcal{L}_{\hat{X^n}_s}, a)\hat{q}^n_s(\mathrm{d}a) \mathrm{d}s + \int_0^T c(s, \hat{X^n}_s, \mathcal{L}_{\hat{X^n}_s})\mathrm{d}\hat{K}^n_s + g(\hat{X}^n_T, \mathcal{L}_{\hat{X^n}_T}) \right] \\
&= \hat{\mathbb{E}}\left[ \int_0^T \int_A^{} f(s, \hat{X}_s,\mathcal{L}_{\hat{X}_s}, a)\hat{q}_s(\mathrm{d}a) \mathrm{d}s + \int_0^T c(s, \hat{X}_s, \mathcal{L}_{\hat{X}_s})\mathrm{d}\hat{K}_s + g(\hat{X}_T, \mathcal{L}_{\hat{X}_T}) \right].
\end{align*}
Theorem \ref{theo1} is thus proved.
\end{proof}
\subsection{Proof of Corollary \ref{theo2}}  
\begin{proof}
Let  
\begin{equation*}
\hat{r}=(\hat{\Omega}, \hat{\mathcal{F}},\hat{\mathbb{F}}, \hat{\mathbb{P}}, \hat{X}, \hat{B}, \hat{K}, \hat{q})
\end{equation*}
be the optimal relaxed control constructed in the previous proof.

We define, for each $(t, \omega) \in [0,T] \times \hat{\Omega}$,
\begin{equation*}
\hat{c}(t,\omega) := \left( \int_A b(t, \hat{X}_t,\mathcal{L}_{\hat{X}_t}, a)\,\hat{q}_t(\mathrm{d}a), \; \int_A f(t, \hat{X}_t,\mathcal{L}_{\hat{X}_t}, a)\,\hat{q}_t(\mathrm{d}a) \right).
\end{equation*}
By Assumption~\textbf{(A.4)}, the set $\mathcal{S}(t,x, \mu)$ is closed and convex, and hence 
\begin{equation*}
\hat{c}(t,\omega) \in \mathcal{S}(t, \hat{X}_t, \mathcal{L}_{\hat{X}_t}).
\end{equation*}
Applying the measurable selection theorem (Theorem~\textbf{A.9} in~\cite{hauss}), there exists an $A$-valued progressively measurable process $\hat{u}$ such that, $\hat{\mathbb{P}}$-almost surely, for all $t \in [0,T]$,
\begin{equation*}
\hat{c}(t,\omega) = \left( b\bigl(t, \hat{X}(t,\omega),\mathcal{L}_{\hat{X}_t}. \hat{u}(t,\omega)\bigr), \; f\bigl(t, \hat{X}(t,\omega),\mathcal{L}_{\hat{X}_t}, \hat{u}(t,\omega)\bigr) \right).
\end{equation*}
Consequently, for every $t \in [0,T]$,
\begin{equation*}
\int_A f(t, \hat{X}_t,\mathcal{L}_{\hat{X}_t}, a)\,\hat{q}_t(\mathrm{d}a) = f\bigl(t, \hat{X}_t, \mathcal{L}_{\hat{X}_t},\hat{u}_t\bigr),
\qquad
\int_A b(t, \hat{X}_t,\mathcal{L}_{\hat{X}_t}, a)\,\hat{q}_t(\mathrm{d}a) = b\bigl(t, \hat{X}_t,\mathcal{L}_{\hat{X}_t}, \hat{u}_t\bigr).
\end{equation*}

Finally, the pair $(\hat{X}, \hat{K})$ satisfies the reflected stochastic differential system
\begin{equation*}
\begin{cases}
\mathrm{d}\hat{X}_t = b(t, \hat{X}_t, \mathcal{L}_{\hat{X}_t},\hat{u}_t)\,\mathrm{d}t + \sigma(t, \hat{X}_t, \mathcal{L}_{\hat{X}_t})\,\mathrm{d}\hat{B}_t + \mathrm{d}\hat{K}_t, & t \in [0,T], \\[4pt]
\hat{X}_t \ge 0, \quad \hat{\mathbb{P}}\text{-a.s.}, \\[4pt]
\displaystyle \int_0^T \hat{X}_t\,\mathrm{d}\hat{K}_t = 0, \quad \hat{\mathbb{P}}\text{-a.s.}
\end{cases}
\end{equation*}

Since the cost functional $J$ depends on the control only through the functions $b$ and $f$, it follows that
\begin{equation*}
J(\hat{r}) = J(\hat{\alpha}),
\end{equation*}
where
\begin{equation*}
\hat{\alpha} = (\hat{\Omega}, \hat{\mathcal{F}}, \hat{\mathbb{F}}, \hat{\mathbb{P}}, \hat{X}, \hat{B}, \hat{K}, \hat{u}).
\end{equation*}
Therefore, $\hat{\alpha}$ is an \emph{optimal strict control} associated with the \emph{optimal relaxed control} $\hat{r}$.
\end{proof}
\subsection{Proof of Theorem \ref{approximatio1}}
To prove the approximation Theorem \ref{approximatio1}, we begin by stating the chattering lemma as established in \cite{meleard}, page 196.

\begin{lemma}\label{chattering lemma}
Let $q$ be a relaxed control defined on a filtered probability space $(\Omega, \mathcal{F}, \mathbb{F}, \mathbb{P})$. Then there exists a sequence of $\mathbb{F}$-adapted processes $(u^n)$, taking values in $A$, such that the sequence of random measures $\delta_{u^n_t}(\mathrm{d}a)\, \mathrm{d}t$ converges almost surely to $q_t(\mathrm{d}a)\, \mathrm{d}t$ in $\mathcal{V}$.
\end{lemma}
Let 
\begin{equation*}
r = (\Omega, \mathcal{F}, \mathbb{F}, \mathbb{P}, B, X, K, q)
\end{equation*}
be a relaxed control, where $(X, K)$ denotes the solution of the RMVSDE~\eqref{def of Relaxed} associated with the control $q$. 
Let $(u^n)_{n \ge 1}$ be a sequence of strict controls approximating the relaxed control $q$, as given by Lemma~\ref{chattering lemma}. 
Denote by $(X^n, K^n)$ the solution of the RMVSDE~\eqref{def of Strict Control}, driven by the Brownian motion $B$, corresponding to the control $u^n$ on the probability space $(\Omega, \mathcal{F}, \mathbb{F}, \mathbb{P})$.\\  
We then define the sequence of admissible strict controls  
\begin{equation*}
\alpha^n := (\Omega, \mathcal{F}, \mathbb{F}, \mathbb{P}, B, X^n, K^n, u^n).
\end{equation*}

\begin{lemma}\label{adjlemma3}
 the following results hold:
\begin{description}
    \item[i)] \(\displaystyle \lim_{n \to \infty} \mathbb{E}^{\mathbb{P}} \left( \sup_{0 \leq t \leq T} |X^n_t - X_t|^2  + \sup_{0 \leq t \leq T} |K^n_t - K_t|^2 \right ) = 0\);
    \item[ii)] \(\displaystyle \lim_{n \to \infty} J(\alpha^n) = J(r)\).
\end{description}
\end{lemma}

\begin{proof}
Applying Itô’s formula to the continuous semimartingale $ |X^n_t - X_t|^2 $, we obtain
\begin{align*}
\mathrm{d}|X^n_t - X_t|^2 
&= 2 (X^n_t - X_t)\Big(b(t,X^n_t,\mathcal{L}_{X^n_t},u^n_t)-\!\int_A b(t,X_t,\mathcal{L}_{X_t},a)q_t(\mathrm{d}a)\Big)\mathrm{d}t  \\
&\quad + 2 (X^n_t - X_t)\left(\sigma(t,X^n_t,\mathcal{L}_{X^n_t})-\sigma(t,X_t,\mathcal{L}_{X_t})\right)\mathrm{d}B_t \\
&\quad + \left(\sigma(t,X^n_t,\mathcal{L}_{X^n_t})-\sigma(t,X_t,\mathcal{L}_{X_t})\right)\left(\sigma(t,X^n_t,\mathcal{L}_{X^n_t})-\sigma(t,X_t,\mathcal{L}_{X_t})\right)^*\mathrm{d}t \\
&\quad + 2 (X^n_t - X_t)\mathrm{d}K^n_t - 2 (X^n_t - X_t)\mathrm{d}K_t.
\end{align*}

By the Skorokhod conditions, we have $(X^n_t - X_t)\mathrm{d}K^n_t \leq 0$ and $(X^n_t - X_t)\mathrm{d}K_t \geq 0$. Therefore,
\begin{align*}
\mathrm{d}|X^n_t - X_t|^2 
&\leq 2 (X^n_t - X_t)\Big(b(t,X^n_t,\mathcal{L}_{X^n_t},u^n_t)-\!\int_A b(t,X_t,\mathcal{L}_{X_t},a)q_t(\mathrm{d}a)\Big)\mathrm{d}t \\
&\quad + 2 (X^n_t - X_t)\left(\sigma(t,X^n_t,\mathcal{L}_{X^n_t})-\sigma(t,X_t,\mathcal{L}_{X_t})\right)\mathrm{d}B_t \\
&\quad + \left(\sigma(t,X^n_t,\mathcal{L}_{X^n_t})-\sigma(t,X_t,\mathcal{L}_{X_t})\right)\left(\sigma(t,X^n_t,\mathcal{L}_{X^n_t})-\sigma(t,X_t,\mathcal{L}_{X_t})\right)^*\mathrm{d}t.
\end{align*}

Taking the supremum over $ [0,T] $ and the expectation, we obtain
\begin{align*}
\mathbb{E}\!\left[\sup_{0 \leq s \leq T} |X^n_s - X_s|^2 \right]
&\leq 2\,\mathbb{E}\!\left[\int_0^T |(X^n_t - X_t)\big(b(t,X^n_t,\mathcal{L}_{X^n_t},u^n_t)-\!\int_A b(t,X_t,\mathcal{L}_{X_t},a)q_t(\mathrm{d}a)\big)|\,\mathrm{d}t\right] \\
&\quad + 2\,\mathbb{E}\!\left[\sup_{0 \leq s \leq T}\left|\int_0^s (X^n_t - X_t)\left(\sigma(t,X^n_t,\mathcal{L}_{X^n_t})-\sigma(t,X_t,\mathcal{L}_{X_t})\right)\mathrm{d}B_t\right|\right] \\
&\quad + \mathbb{E}\!\left[\int_0^T |\sigma(t,X^n_t,\mathcal{L}_{X^n_t})-\sigma(t,X_t,\mathcal{L}_{X_t})|^2\mathrm{d}t\right].
\end{align*}

Using the Lipschitz property of $\sigma$ and the Burkholder--Davis--Gundy inequality, we get
\begin{align*}
&2\,\mathbb{E}\!\left[\sup_{0 \leq s \leq T}\left|\int_0^s (X^n_t - X_t)\left(\sigma(t,X^n_t,\mathcal{L}_{X^n_t})-\sigma(t,X_t,\mathcal{L}_{X_t})\right)\mathrm{d}B_t\right|\right]\\
&\leq 2C\,\mathbb{E}\!\left[\int_0^T |X^n_t - X_t|^2|\sigma(t,X^n_t,\mathcal{L}_{X^n_t})-\sigma(t,X_t,\mathcal{L}_{X_t})|^2\mathrm{d}t\right]^{1/2} \\
&\leq \varepsilon\,\mathbb{E}\!\left[\sup_{0 \leq s \leq T} |X^n_s - X_s|^2\right]
 + \frac{1}{\varepsilon}\,\mathbb{E}\!\left[\int_0^T |\sigma(t,X^n_t,\mathcal{L}_{X^n_t})-\sigma(t,X_t,\mathcal{L}_{X_t})|^2\mathrm{d}t\right],
\end{align*}
where we used the standard inequality $2ab \leq \varepsilon a^2 + \frac{b^2}{\varepsilon}$.

Similarly,
\begin{align*}
&2\,\mathbb{E}\!\left[\int_0^T |(X^n_t - X_t)\big(b(t,X^n_t,\mathcal{L}_{X^n_t},u^n_t)-\!\int_A b(t,X_t,\mathcal{L}_{X_t},a)q_t(\mathrm{d}a)\big)|\,\mathrm{d}t\right]\\
&\leq \varepsilon\,\mathbb{E}\!\left[\sup_{0 \leq s \leq T} |X^n_s - X_s|^2\right]\quad + \frac{1}{\varepsilon}\,\mathbb{E}\!\left[\left(\int_0^T \big|b(t,X^n_t,\mathcal{L}_{X^n_t},u^n_t)-\!\int_A b(t,X_t,\mathcal{L}_{X_t},a)q_t(\mathrm{d}a)\big|\,\mathrm{d}t\right)^2\right].
\end{align*}

Combining the two inequalities above and choosing $\varepsilon > 0$ small enough, we obtain
\begin{align*}
\mathbb{E}\!\left[\sup_{0 \leq s \leq T} |X^n_s - X_s|^2 \right]
&\leq C\,\mathbb{E}\!\left[\int_0^T |X^n_t - X_t|^2\mathrm{d}t\right] \\
&\quad + C\,\mathbb{E}\!\left[\left(\int_0^T \big|b(t,X^n_t,\mathcal{L}_{X^n_t},u^n_t)-\!\int_A b(t,X_t,\mathcal{L}_{X_t},a)q_t(\mathrm{d}a)\big|\,\mathrm{d}t\right)^2\right].
\end{align*}

Finally, by the Lipschitz continuity of $b$, we have
\begin{align*}
\mathbb{E}\!\left[\sup_{0 \leq s \leq T} |X^n_s - X_s|^2 \right]
&\leq C\,\mathbb{E}\!\left[\int_0^T |X^n_t - X_t|^2\mathrm{d}t\right] \\
&\quad + C\,\mathbb{E}\!\left[\left(\int_0^T \big|b(t,X_t,\mathcal{L}_{X^n_t},u^n_t)-\!\int_A b(t,X_t,\mathcal{L}_{X_t},a)q_t(\mathrm{d}a)\big|^2\,\mathrm{d}t\right)\right]\\
&=: C\,\mathbb{E}\!\left[\int_0^T |X^n_t - X_t|^2\mathrm{d}t\right] + CI^n.
\end{align*}

We have, 
\begin{align*}
I^n =& \mathbb{E} \left(\int_{0}^{T}\mathds{1}_{{\{|X_{t}|> M\}}} \left| \int_{A}b(t, X_t, \mathcal{L}_{X_t},a)q_t(\mathrm{d}a)-\int_{A}^{}b(t, X_t,\mathcal{L}_{X_t}, a) \delta_{u^n_t}(\mathrm{d}a)\right|^2 \mathrm{d}t \right)\\
&+\mathbb{E} \left(\int_{0}^{T}\mathds{1}_{{\{|X_{t}|\leq M\}}}^{} \left| \int_{A}b(t, X_t,\mathcal{L}_{X_t}, a)q_t(\mathrm{d}a)-\int_{A}^{}b(t, X_t,\mathcal{L}_{X_t}. a) \delta_{u^n_t}(\mathrm{d}a)\right|^2 \mathrm{d}t \right)\\
&:= I^n_1 + I^n_2.
\end{align*}
Since the mapping \((s,a) \mapsto b(t, X_t,\mathcal{L}_{X_t}, a) \mathds{1}_{\{|X_t| \leq M\}}\) is bounded and measurable, and continuous in $a$, it follows from Lemma~\ref{chattering lemma} that
\begin{equation*}
\mathds{1}_{\{|X_t| \leq M\}} \left| \int_{A} b(t, X_t,\mathcal{L}_{X_t}, a) \, q_t(\mathrm{d}a) - \int_{A} b(t, X_t,\mathcal{L}_{X_t}, a) \, \delta_{u_t^n}(\mathrm{d}a) \right|^2 \to 0 \quad \text{a.s.}.
\end{equation*}
Moreover, by assumption \textbf{(A.1)}, the above expression is uniformly bounded as
\begin{equation*}
\mathds{1}_{\{|X_t| \leq M\}} \left| \int_{A} b(t, X_t,\mathcal{L}_{X_t}, a) \, q_t(\mathrm{d}a) - \int_{A} b(t, X_t,\mathcal{L}_{X_t}, a) \, \delta_{u_t^n}(\mathrm{d}a) \right|^2 \leq C(1 + M^2).
\end{equation*}
Applying the Dominated Convergence Theorem with respect to the product measure $\mathbb{P} \times \mathrm{d}t$, we conclude that $I_2^n \to 0$.

On the other hand, using the quadratic growth condition on $b$, we obtain the following bound
\begin{align*}
I^n_1 &\leq 2\mathbb{E} \left( \int_0^T \mathds{1}_{\{|X_t| > M\}} \left| \int_{A} b(t, X_t,\mathcal{L}_{X_t}, a)\, q_t(\mathrm{d}a) \right|^2 + \left| \int_{A} b(t, X_t,\mathcal{L}_{X_t}, a)\, \delta_{u^n_t}(\mathrm{d}a) \right|^2 \mathrm{d}t \right)\\
&\leq  2 \mathbb{E} \left( \int_0^T \mathds{1}_{\{|X_t| > M\}} (1 + |X_t|^2) \, \mathrm{d}t \right).
\end{align*}
Note that
\begin{equation*}
\mathds{1}_{\{|X_t| > M\}} (1 + |X_t|^2) \xrightarrow[M \to \infty]{} 0 \quad \mathbb{P}\times \mathrm{d}t.\text{a.s.},
\end{equation*}
and this expression is dominated by the integrable function $1 + |X_t|^2 \in L^1(\mathbb{P} \times \mathrm{d}t) $. Therefore, by the Dominated Convergence Theorem,
\begin{equation*}
\lim_{M \to \infty} \mathbb{E} \left( \int_0^T \mathds{1}_{\{|X_t| > M\}} (1 + |X_t|^2) \, \mathrm{d}t \right) = 0.
\end{equation*}

Since this upper bound is independent of $ n$, we can pass to the limit in both $ M \to \infty $ and $n \to \infty $, concluding that
\begin{equation*}
I^n_1 \to 0.
\end{equation*}

Finally, by applying Gronwall’s inequality, we obtain the desired result
\begin{equation}
\mathbb{E}\!\left[\sup_{0 \leq s \leq T} |X^n_s - X_s|^2 \right] \to 0.
\end{equation}

On the other hand, we have 
\begin{equation*}
K^n_s - K_s 
= \int_{0}^{s} \big(b(t,X^n_t,\mathcal{L}_{X^n_t},u^n_t) - \!\int_A b(t,X_t,\mathcal{L}_{X_t},a)q_t(\mathrm{d}a)\big)\mathrm{d}t 
+ \int_{0}^{s} \big(\sigma(t,X^n_t,\mathcal{L}_{X^n_t}) - \sigma(t,X_t,\mathcal{L}_{X_t})\big)\mathrm{d}B_t.
\end{equation*}
Using the same arguments as in the previous part, we conclude that point \textbf{(i)} holds.

For point \textbf{(ii)}, we have
\begin{align*}
J(\alpha^n) - J(r) 
=&\ \mathbb{E} \Bigg[ \int_{0}^{T} \Big( f(t, X^n_t,\mathcal{L}_{X^n_t}, u^n_t) - \int_{A} f(t, X_t,\mathcal{L}_{X_t}, a)\, q_t(\mathrm{d}a) \Big) \mathrm{d}t \\
&\quad +\big( \int_{0}^{T} c(t, X^n_t,\mathcal{L}_{X^n_t})\, \mathrm{d}K^n_t - \int_{0}^{T} c(t, X_t,\mathcal{L}_{X_t})\mathrm{d}K_t\big) 
+ \big( g(X^n_T,\mathcal{L}_{X^n_T}) - g(X_T,\mathcal{L}_{X_T}) \big) \Bigg].
\end{align*}

Using the result established in point \textbf{(i)} and the same line of reasoning as in the previous subsection, we conclude that
\begin{equation*}
J(\alpha^n) - J(r) \longrightarrow 0 \quad \text{as } n \to \infty.
\end{equation*}
\end{proof}

\begin{proof}[Proof of Theorem~\ref{approximatio1}]
Let $ \hat{r} $ be the optimal relaxed control guaranteed by Theorem~\ref{theo2}. 
Then, by Lemma~\ref{adjlemma3}, there exists a sequence of strict controls $ \{\hat{\alpha}^n\}_{n \ge 1} $ such that
\begin{equation*}
\lim_{n \to \infty} J(\hat{\alpha}^n) = J(\hat{r}) = \inf_{r \in \mathcal{R}} J(r) \leq \inf_{\alpha \in \mathcal{U}} J(\alpha).
\end{equation*}
This completes the proof.
\end{proof}

\section*{Appendix}

\begin{lemma}(Theorem 7.3 \cite{belg}).
\label{Tights}
Let $\{P_n\}$ be probability measures on $\mathcal{C} :=\mathcal{C}([0,T], \mathbb{R})$, the space of continuous functions on $[0,T]$, equipped with the uniform topology.\\
The sequence $\{P_n\}$ is tight if and only if the following two conditions are satisfied:
\begin{enumerate}
\item For every positive $ \nu $, there exist constants $a$ and $n_0$ such that a specific condition holds for all $n \geq n_0$.
\begin{equation}
P^n(x : |x(0)| \geq a ) \leq \nu.
\end{equation}
\item For every positive $\epsilon$ and $\nu $, there exist a constant $\delta$, where $0 < \delta < 1$, and a value $n_0$ such that another condition holds for all $n \geq n_0$,
\begin{equation}
P^n(x : \sup_{|t-s| \leq \delta} |x(t)-x(s)| \geq \epsilon ) \leq \nu.
\end{equation}
\end{enumerate}
\end{lemma}
\begin{lemma}(Skorokhod Selection Theorem \cite{ikeda}) \label{lemma-representation} 
Let $E, \mathcal{E})$ be a complete separable metric space, and let $P$ and $P_n$, $n = 1, 2, \dots$, be probability measures on $(E, \mathcal{B}(E))$ such that $P_n$ converges weakly to $P$ as $n \to \infty$. Then, on a probability space $(\tilde{\Omega}, \tilde{\mathcal{F}}, \tilde{P})$, there exist $E$-valued random variables $X_n$, $n = 1, 2, \dots$, and $X$ such that:
\begin{enumerate}
\item \(P = \tilde{P} \circ X^{-1}\),
\item \(P_n = \tilde{P} \circ X_n^{-1}\), for each $n \geq 1$,
\item \(X_n \xrightarrow{n \to \infty} X\), \(\tilde{P}\)-a.s.
\end{enumerate}
\end{lemma}

\end{document}